\documentclass[11pt]{article}
\usepackage{amsmath,amsthm,amsfonts}

\newcommand{\nexteq}{\displaybreak[0]\\ &=}
\newcommand{\nnexteq}{\notag\displaybreak[0]\\ &=}

\newtheorem{thm}{Theorem}
\newtheorem{lem}{Lemma}

\newcommand{\C}{\mathbb{C}}
\newcommand{\R}{\mathbb{R}}

\title{Complex Hadamard matrices attached to even orthogonal schemes of class $4$}
\author{Takuya Ikuta and Akihiro Munemasa}

\begin{document}
\maketitle

\begin{abstract}
A complex Hadamard matrix is a square matrix $W$ with complex 
entries of absolute value $1$ satisfying $WW^\ast= nI$, 
where $\ast$ stands for the Hermitian transpose and $I$ is the 
identity matrix of order $n$. 
In this paper, we give constructions of complex Hadamard matrices
in the Bose--Mesner algebra of a certain 4-class symmetric association scheme. 
Moreover, we determine the Nomura algebras to show that the resulting matrices are
not decomposable into nontrivial generalized tensor products.
\end{abstract}

\section{Introduction}\label{sec:Intro}
A complex Hadamard matrix is a square matrix $W$ with complex 
entries of absolute value $1$ satisfying
$WW^\ast= nI$, 
where $\ast$ stands for the Hermitian transpose and $I$ is the identity matrix of order $n$.
They are the natural generalization of real Hadamard matrices. 
Complex Hadamard matrices appear frequently in various branches of mathematics and 
quantum physics.

A type-II matrix, or an inverse orthogonal matrix,
is a square matrix $W$ with nonzero complex entries satisfying
$W{W^{(-)}}^\top=nI$, where $W^{(-)}$ denotes the entrywise inverse of $W$. 
Obviously, a complex Hadamard matrix is a
type-II matrix.

In \cite{MI}, we gave a method to find a complex Hadamard matrix in the Bose--Mesner algebra
of a symmetric association scheme. Applying this result, we classified complex Hadamard matrices
in the Bose--Mesner algebra of a certain 3-class association scheme.
In this paper, we construct certain complex Hadamard matrices 
in the Bose--Mesner algebra of a 4-class association scheme $(X,\{R_i\}_{i=0}^4)$ with the first eigenmatrix:
\begin{equation} \label{P}
P=\begin{bmatrix}
1&\frac12(q-2)q^{2m-1}&\frac12q^{2m}&q(q^{2m-2}-1)&q-2 \\
1&\frac12(q-2)q^{m-1}&\frac12q^{m}&-(q-1)(q^{m-1}+1)&q-2 \\
1&-\frac12(q-2)q^{m-1}&-\frac12q^{m}&-(q-1)(q^{m-1}-1)&q-2 \\
1&\frac12q^m&-\frac12q^m&0&-1 \\
1&-\frac12q^m&\frac12q^m&0&-1
\end{bmatrix},
\end{equation}
where $q$ and $m$ are positive integers with $q\geq 4$ and $m\geq2$.
Then $|X|=q^{2m}-1$, 
$R_4$ is a disconnected relation, and $R_2$ defines a strongly regular graph.
If $q$ is a power of $2$, an even orthogonal scheme 
is an example of an association scheme with the first eigenmatrix (\ref{P}) (see \cite[Chapter 12.1]{BCN}).
If $m=1$, then $R_3=\emptyset$, and 
this scheme reduces to an even orthogonal scheme of class $3$
which we considered in \cite{MI}.

For a type-II matrix $W\in M_X(\C)$ and $a,b\in X$, we define column vectors
$Y_{ab}$ by setting 
\[
(Y_{ab})_x=
\frac{W_{xa}}{W_{xb}}\quad(x\in X).
\]
The {\em Nomura algebra} $N(W)$ of $W$ is the algebra of matrices
in $M_n(\C)$ such that $Y_{ab}$ is an eigenvector for all $a,b\in X$.
It is shown in \cite[Theorem 1]{JMN} that the Nomura algebra is 
a Bose--Mesner algebra. 

Throughout this paper, we denote by $\frak{X}=(X,\{R_i\}_{i=0}^4)$ 
a symmetric association scheme with the first eigenmatrix (\ref{P}).
Let $A_0,A_1,A_2,A_3,A_4$ be the adjacency matrices of $\frak{X}$.
Let $w_0=1,w_1,w_2,w_3,w_4$ be nonzero complex numbers, and
set
\begin{equation}\label{W} 
W=\sum_{j=0}^4w_j A_j,
\end{equation}
\begin{equation} \label{aij0802}
a_{i,j}=\frac{w_i}{w_j}+\frac{w_j}{w_i} \quad (0\leq i<j\leq 4).
\end{equation}

The main purpose of this paper is to prove the following:
\begin{thm} \label{thm:main}
Assume that
\begin{equation} \label{assump}
w_4=1.
\end{equation}
\begin{enumerate}
 \item[\rm{(i)}] Assume $w_1=1$. 
 Then, the matrix $W$ in {\rm (\ref{W})} is a complex Hadamard matrix if and only if
 \[
 w_2^2+\frac{2(q^{2m}-2)}{q^{2m}}w_2+1=0 \quad \text{and} \quad w_3=1.
 \]
 \item[\rm{(ii)}] Assume
  \begin{equation} \label{eq:1024}
  a_{0,1}=\frac{2(q^{4m-2}-(q+2)q^{2m-1}+2)}{(q^{2m-1}+q-2)q^{2m-1}}.
  \end{equation}
 Then, the matrix $W$ in {\rm (\ref{W})} is a complex Hadamard matrix if and only if
  \begin{align}
  w_2=&-\frac{(q-1)q^{2m-1}w_1+q^{2m-1}+q-2}{(q^{2m-1}-1)q}, \label{eq:3-2} \\
  w_3=&1. \nonumber
  \end{align}
\end{enumerate}
\end{thm}

\begin{thm} \label{thm:main2}
Let $W$ be a complex Hadamard matrix given in {\rm(i)} and {\rm(ii)} of Theorem~{\rm\ref{thm:main}}.
The algebra $N(W)$ coincides with the linear span of $I$ and $J$.
In particular, $W$ is not equivalent to a nontrivial generalized
tensor product.
\end{thm}

The reason for the assumption (\ref{assump}) is as follows: 
Calculating the conditions under which the matrix (\ref{W}) becomes 
a complex Hadamard matrix experimentally for small $q$ and $m$,
we find that (\ref{assump}) is fulfilled, or
\begin{enumerate}
\setcounter{enumi}{2}
\item $a_{0,4}=2(q^{2m}-6)/(q^{2m}-4)$, or
\item $a_{0,4}$ is a zero of a polynomial of degree $9$.
\end{enumerate}

For the case (iii) with $m=2,3$, we have $w_1=w_3=w_4$. Therefore, this case reduces to the case in which
the matrix $W$ given in (\ref{W}) belongs to the Bose--Mesner algebra 
of the strongly regular graph defined by $R_2$.
However, it seems to be difficult to prove $w_1=w_3=w_4$ for arbitrary $m\geq 4$.

For the case (iv), we verified that the polynomial in $a_{0,4}$ of degree $9$ is an irreducible polynomial 
for $m=2,\ldots,9$ and $q=2^s$ with $2\leq s\leq 10000$.
However, it seems difficult to determine the polynomial of degree $9$ satisfied by $a_{0,4}$ in general.
For example, 
for $(q,m)=(4,2)$,
if the matrix (\ref{W}) is a complex Hadamard matrix,
then $a_{0,4}$ is a zero of the polynomial
\begin{align*}
p(x)=&x^9-\frac{235721}{1785}x^8-\frac{17957726593}{62475}x^7+\frac{33219815829811}{937125}x^6 \\
&-\frac{12554318926285933}{4685625}x^5+\frac{29740292638491103}{312375}x^4 \\
&-\frac{696525696876795217}{187425}x^3+\frac{851886544261448041}{37485}x^2 \\
&-\frac{124583919439776136}{2499}x+\frac{30888835313436500}{833}.
\end{align*}
It can be shown that $p(x)$ has only one real root in $(-2,2)$ by using Sturm's theorem.
Then, by using Lemmas~\ref{thm:1} and \ref{thm:2} below,
there exist $w_1,w_2,w_3,w_4$ such that (\ref{W}) is a complex Hadamard matrix.

Under the hypothesis of (\ref{assump}), we find that $a_{0,1}=2$ or $a_{0,1}$ is given by (\ref{eq:1024}), or
\begin{enumerate}
\setcounter{enumi}{4}
\item $a_{0,1}$ is a zero of a polynomial of degree $4$.
\end{enumerate}
It seems to be difficult to determine $w_1,w_2,w_3$ for the case (vi) for arbitrary $q$ and $m$.
For example, for $(q,m)=(4,2)$,
if the matrix (\ref{W}) is a complex Hadamard matrix,
then $w_1,w_2,w_3$ are given by the following:
\begin{align*}
 & w_1^2+\frac{21s-7140\pm 85t}{176}w_1+1=0, \\
 & w_2=-\frac{64(w_1^2-1)}{127w_1+64a_{0,2}}, \\
 & w_3=\frac{90(w_1^2-1)}{90a_{1,3}w_1-4s+1117}, \\
 & a_{0,2}=\frac{43s-14620\pm 85t}{352}, \\
 & a_{1,3}=\frac{21s-1848\mp(4s+1253)t}{2640}, \\
 & s=\sqrt{104899}, \\
 & t^2=\frac{8s-2591}{3}.
\end{align*}




\section{Preliminaries}\label{sec:2}
We define a polynomial in three indeterminates $X,Y,Z$ as follows:
\[
g(X,Y,Z)=X^2+Y^2+Z^2-XYZ-4.
\]
We define a polynomial in six indeterminates 
$X_{0,1},X_{0,2},X_{0,3},X_{1,2},X_{1,3},X_{2,3}$ as follows:
\[
h(X_{0,1},X_{0,2},X_{0,3},X_{1,2},X_{1,3},X_{2,3})
=\det\begin{bmatrix}
2&X_{0,1}&X_{0,2}\\
X_{0,1}&2&X_{1,2}\\
X_{0,3}&X_{1,3}&X_{2,3}
\end{bmatrix}.
\]
For a finite set $N$ and a positive integer $k$, 
we denote by $\binom{N}{k}$ the collection of all $k$-element subsets
of $N$.

\begin{lem}[{\cite[Lemma 4]{MI}}] \label{thm:1}
Let $N=\{0,1,\dots,d\}$, 
$N_3=\binom{N}{3}$ and
$N_4=\binom{N}{4}$.
Let $a_{i,j}$ $(0\leq i,j\leq d,\;i\neq j)$ be
complex numbers satisfying
\begin{align}
a_{i,j}&=a_{j,i}\quad(0\leq i<j\leq d),\label{aij1}\\
g(a_{i,j},a_{j,k},a_{i,k})&=0\quad(\{i,j,k\}\in N_3),\label{N1}\\
h(a_{i,j},a_{i,k},a_{i,\ell},a_{j,k},a_{j,\ell},a_{k,\ell})&=0
\quad(\{i,j,k,\ell\}\in N_4).\label{N4}
\end{align}
Assume 
\begin{equation}\label{a1}
a_{i_0,i_1}\neq\pm2\quad\text{for some $i_0,i_1$ with
$0\leq i_0<i_1\leq d$.}
\end{equation}
Let $w_{i_0}$, $w_{i_1}$ be nonzero complex numbers satisfying
\begin{equation}\label{az1d}
\frac{w_{i_0}}{w_{i_1}}+\frac{w_{i_1}}{w_{i_0}}=a_{{i_0,i_1}}.
\end{equation}
Then for
complex numbers $w_i$ $(0\leq i\leq d,\;i\neq i_0,i_1)$,
the following are equivalent:
\begin{enumerate}
\item for all $i,j$ with $0\leq i,j\leq d$ and $i\neq j$,
\begin{equation}\label{az2}
\frac{w_j}{w_i}+\frac{w_i}{w_j}=a_{i,j}
\end{equation}
\item
for all $i,j$ with $0\leq i\leq d$, $i\neq i_0,i_1$,
\begin{equation}\label{azid}
w_i=
\frac{w_{i_1}^2-w_{i_0}^2}{a_{{i_1},i}w_{i_1}-a_{{i_0},i}w_{i_0}}.
\end{equation}
\end{enumerate}
Moreover, if one of the two equivalent conditions
{\rm(i), (ii)} is satisfied,
$a_{i,j}$ $(0\leq i<j\leq d)$ are all real and
\begin{equation}\label{-22}
-2<a_{i_0,i_1}<2, 
\end{equation}
then $|w_i|=|w_j|$ for $0\leq i<j\leq d$.
\end{lem}

We let $\mathcal{A}$ denote a
symmetric Bose--Mesner algebra with adjacency matrices
$A_0,A_1,\dots,A_d$. Let $n$ be the size of the matrices $A_i$,
and we denote by 
\[
P=(P_{i,j})_{\substack{0\leq i\leq d\\ 0\leq j\leq d}}
\]
the first eigenmatrix of $\mathcal{A}$. Then the 
adjacency matrices are expressed as
\[
A_j=\sum_{i=0}^d P_{i,j}E_i\quad(j=0,1,\dots,d),
\]
where $E_0=\frac{1}{n}J,E_1,\dots,E_d$ are the primitive idempotents
of $\mathcal{A}$.

Let $w_0,w_1,\dots,w_d$ be nonzero complex numbers, and
set
\begin{equation}\label{eq:W} 
W=\sum_{j=0}^dw_j A_j\in\mathcal{A}.
\end{equation}


\begin{lem}[{\cite[Lemma 7]{MI}}] \label{thm:2}
Let $X_{i,j}$ $(0\leq i<j\leq d)$ be indeterminates and
let $e_k$ be the polynomial defined by
\begin{equation} \label{eq:04}
e_k=
\sum_{0\leq i<j\leq d} P_{k,i}P_{k,j}X_{i,j}
+\sum_{i=0}^d P_{k,i}^2-n
\quad(k=1,\dots,d).
\end{equation}
Let $a_{i,j}$ $(0\leq i,j\leq d,\;i\neq j)$ and
$w_i$ $(0\leq i\leq d)$ be complex numbers.
Assume that $w_i\neq0$ for all $i$ with $0\leq i\leq d$
and that {\rm(\ref{az2})} holds.
Then the following statements are equivalent:
\begin{enumerate}
\item 
the matrix $W$ given by {\rm(\ref{eq:W})}
is a type-II matrix,
\item
$(a_{i,j})_{0\leq i<j\leq d}$
is a common zero of $e_k$ $(k=1,\dots,d)$.
\end{enumerate}
Moreover, if one of the two equivalent conditions
{\rm(i), (ii)} is satisfied, 
$a_{i,j}\in\R$ for all $i,j$ with $0\leq i<j\leq d$,
and {\rm(\ref{-22})} holds for some $i_0, i_1$ with
$0\leq i_0<i_1\leq d$, then 
$W$ is a scalar multiple of a complex
Hadamard matrix.
\end{lem}

\bigskip

We now describe the proof of Theorem~\ref{thm:main} briefly.
Let $A_0,A_1,A_2,A_3,A_4$ be the adjacency matrices of an association scheme $\frak{X}$
with the first eigenmatrix (\ref{P}).
Let $w_0=1,w_1,w_2,w_3,w_4$ be nonzero complex numbers, and
$W$ be the matrix defined by (\ref{W}).
For $i,j\in\{0,1,2,3,4\}$, define $a_{i,j}$ by (\ref{aij0802}).
We write 
\begin{equation} \label{aa}
\boldsymbol{a}=(a_{0,1},a_{0,2},a_{0,3},a_{0,4},a_{1,2},a_{1,3},a_{1,4},a_{2,3},a_{2,4},a_{3,4})
\end{equation}
for brevity.
Consider the polynomial ring 
\begin{equation} \label{RC}
R=\mathbb{C}[X_{0,1},X_{0,2},X_{0,3},X_{0,4},X_{1,2},X_{1,3},X_{1,4},X_{2,3},X_{2,4},X_{3,4}].
\end{equation}

In Section 3, we first assume that $W$ is a complex Hadamard matrix. 
Then by Lemmas~\ref{thm:1} and \ref{thm:2},
$\boldsymbol{a}$ is a common zero of the polynomials
\begin{align}
& g(X_{i,j},X_{i,k},X_{j,k})\quad(\{i,j,k\}\in\binom{\{0,1,2,3,4\}}{3}),\label{gg} \\
& h(X_{i,j},X_{i,k},X_{i,l},X_{j,k},X_{j,l},X_{k,l})
\quad(\{i,j,k,l\}\in\binom{\{0,1,2,3,4\}}{4}), \label{hh} \\
& e_k\quad (k\in\{1,2,3,4\}). \label{ee}
\end{align}
Let $\mathcal{I}$ be the ideal of $R$ generated by these polynomials.
Calculating the ideal generated by $\mathcal{I}$ and $X_{0,4}-2$,
we find (i) and (ii) of Theorem~\ref{thm:main}.

Conversely, we assume that $w_4=1$ and
$w_1,w_2,w_3$ are given in Theorem~\ref{thm:main}.
Then, to show that the matrix $W$ given in (\ref{W}) is a complex Hadamard matrix,
we check that $\boldsymbol{a}$ defined by (\ref{aij0802}), (\ref{aa})
is a zero of the polynomials (\ref{gg}), (\ref{hh}), (\ref{ee}).
Moreover, we check that $-2<a_{i_0,i_1}<2$ holds
for some $i_0,i_1$ with $0\leq i_0<i_1\leq 4$.

All the computer calculations in this paper were performed with the help of
Magma \cite{magma}.
In order to facilitate computations covering all possible values of
the integer $q$, we perform the computations in
the polynomial ring with $12$ variables
$q,r=q^m$ and $X_{i,j}$ over the field of rational numbers,
rather than the ring \eqref{RC}.
The results valid for this generic setting are also valid for
arbitrary integers $q$ and $m$.

\section{Proof of Theorem~\ref{thm:main}}\label{sec:3}
Recall $q\geq 4$ and $m\geq 2$, and
$\mathcal{I}$ is the ideal of the polynomial ring $R$ generated by 
the polynomials (\ref{gg}), (\ref{hh}), and (\ref{ee}).
For the remainder of this section, we assume that $a_{0,4}=2$, that is, $w_4=1$.
Let $\mathcal{I}_1$ denote the ideal generated by $\mathcal{I}$ and $X_{0,4}-2$.
For Lemmas~\ref{lem:a01}--\ref{lem:a02d} we assume that $\boldsymbol{a}$ defined in (\ref{aa}) 
is a common zero of the polynomials in $\mathcal{I}_1$.

\begin{lem} \label{lem:a01}
We have
\begin{equation} \label{a12}
a_{1,2}=-\frac{2(q^{2m}-2)}{q^{2m}}.
\end{equation}
\end{lem}
\begin{proof}
We can verify that $\mathcal{I}_1$ contains $X_{1,2}+2(q^{2m}-2)/q^{2m}$. Hence we have (\ref{a12}).
\end{proof}

\begin{lem}  \label{lem:a02eq2}
Assume $a_{0,1}=2$.
Then, $(w_1,w_2,w_3)$ is given in {\rm(i)} of Theorem~{\rm \ref{thm:main}}.
\end{lem}
\begin{proof}
Let $\mathcal{I}_2$ denote the ideal generated by $\mathcal{I}_1$ and $X_{0,1}-2$.
Then we can verify that $\mathcal{I}_2$ contains $(X_{0,3}-2)^2$, that is, $a_{0,3}=2$.
Hence $w_1=w_3=1$.
Since $a_{1,2}$ is given in (\ref{a12}),
the matrix $W$ given in (\ref{W}) belongs to 
the Bose--Mesner algebra of the strongly regular graph defined by $R_2$.
From \cite{ChanGodsil} we have the condition of $w$ given in {\rm(i)} of Theorem~{\rm\ref{thm:main}}.
\end{proof}

\begin{lem}  \label{lem:a02d}
Assume that $a_{0,1}$ is given in {\rm(\ref{eq:1024})}.
Then, $(w_1,w_2,w_3)$ is given in {\rm(ii)} of Theorem~{\rm \ref{thm:main}}.
\end{lem}
\begin{proof}
Let $\mathcal{I}_3$ denote the ideal generated by $\mathcal{I}_1$ and $X_{0,1}-a_{0,1}$.
Then we can verify that $\mathcal{I}_3$ contains 
$q(q^{2m-1}+q-2)X_{0,2}+2(q^{2m}-q^2+2q-2)$, that is,
\begin{equation} \label{a02inLem7}
a_{0,2}=-\frac{2(q^{2m}-q^2+2q-2)}{q(q^{2m-1}+q-2)}
\end{equation}
Let $\mathcal{I}_4$ denote the ideal generated by $\mathcal{I}_3$ and $p_1(X_{0,2})$.
Then we can verify that $\mathcal{I}_4$ contains $X_{0,3}-2$, that is, $w_3=1$.
From (\ref{azid}) we obtain
\[
w_2=\frac{w_1^2-1}{a_{1,2}w_1-a_{0,2}}.
\]
Since $w_1^2-a_{0,1}w_1+1=0$,
we have (\ref{eq:3-2}) from (\ref{a12}), (\ref{a02inLem7}).
\end{proof}

\begin{proof}[Proof of Theorem~{\rm\ref{thm:main}}]
Suppose that the matrix (\ref{W}) is a complex Hadamard matrix.
For $i,j\in\{0,1,2,3,4\}$, define $a_{i,j}$ by (\ref{aij0802}).
Let $\boldsymbol{a}$ be given in (\ref{aa}).
Then $\boldsymbol{a}$ is a common zero of the polynomials in $\mathcal{I}_1$ by Lemma~\ref{thm:2}.
From Lemmas~\ref{lem:a02eq2}, \ref{lem:a02d}
we have (i) and (ii) of Theorem~\ref{thm:main}.

Conversely, assume that $w_1,w_2,w_3$, and $w_4$ are given in Theorem~\ref{thm:main}.
Then, we show that the matrix given in (\ref{W}) is a complex Hadamard matrix.
To do this, 
we check that $\boldsymbol{a}$ defined by (\ref{aij0802}) is 
a zero of the polynomials (\ref{gg}), (\ref{hh}), and (\ref{ee}),
and $(w_1,w_2,w_3)$ are complex numbers of absolute value $1$.
The latter condition is satisfied if $-2<a_{i_0,i_1}<2$ holds for some $i_0,i_1$ with $0\leq i_0<i_1\leq 4$.

Case (i) is done by \cite{ChanGodsil}.


Next consider Case (ii).
From (\ref{aij0802}), (\ref{eq:1024}), and (\ref{eq:3-2}) we have (\ref{a12}) and (\ref{a02inLem7}).
Then we have
\[
\boldsymbol{a}=\left( a_{0,1},a_{0,2},2,2,a_{1,2},a_{0,1},a_{0,1},a_{0,2},a_{0,2},2\right).
\]
This is a zero of the polynomials (\ref{gg}), (\ref{hh}), and  (\ref{ee}).
It is easy to check that $0<a_{0,1}<2$.
\end{proof}

\section{Proof of Theorem~\ref{thm:main2}}\label{sec:4}
Since $q^{2m}-1$ is a composite, there are uncountably many
inequivalent complex Hadamard matrices of order $q^{2m}-1$
by \cite{Craigen}. Indeed, such matrices can be constructed
using generalized tensor products \cite{HS}. 
We show that 
 none of our complex Hadamard matrices is equivalent to a nontrivial generalized tensor product.
This is done by showing that the Nomura algebra of our
complex Hadamard matrices has dimension $2$. 
According to
\cite{HS}, the Nomura algebra of a nontrivial generalized tensor product
of type-II matrices is imprimitive, and this is never the case
when it has dimension $2$.

Recall $q\geq 4$ and $m\geq 2$.
The intersection matrices $B_i=(p_{ij}^k)$ ($i=0,\ldots,4$) of $\frak{X}$ are given by the following:
\begin{align*}
B_0&=\begin{bmatrix}
1&0&0&0&0 \\
0&1&0&0&0 \\
0&0&1&0&0 \\
0&0&0&1&0 \\
0&0&0&0&1
\end{bmatrix}, \\
B_1&=\begin{bmatrix}
0&1&0&0&0 \\
\frac{(q-2)q^{2m-1}}{2}&\frac{(q-2)^2q^{2m-2}}{4}&\frac{(q-2)^2q^{2m-2}}{4}&\frac{(q-2)^2q^{2m-2}}{4}&\frac{(q-4)q^{2m-1}}{4} \\
0&\frac{(q-2)q^{2m-1}}{4}&\frac{(q-2)q^{2m-1}}{4}&\frac{(q-2)q^{2m-1}}{4}&\frac{q^{2m}}{4} \\
0&\frac{(q-2)(q^{2m-2}-1)}{2}&\frac{(q-2)(q^{2m-2}-1)}{2}&\frac{(q-2)q^{2m-2}}{2}&0 \\
0&\frac{q-4}{2}&\frac{q-4}{2}&0&0 
\end{bmatrix}, \\
B_2&=\begin{bmatrix}
0&0&1&0&0 \\
0&\frac{(q-2)q^{2m-1}}{4}&\frac{(q-2)q^{2m-1}}{4}&\frac{(q-2)q^{2m-1}}{4}&\frac{q^{2m}}{4} \\
\frac{q^{2m}}{2}&\frac{q^{2m}}{4}&\frac{q^{2m}}{4}&\frac{q^{2m}}{4}&\frac{q^{2m}}{4} \\
0&\frac{q(q^{2m-2}-1)}{2}&\frac{q(q^{2m-2}-1)}{2}&\frac{q^{2m-1}}{2}&0 \\
0&\frac{q}{2}&\frac{q-2}{2}&0&0
\end{bmatrix}, \\
B_3&=\begin{bmatrix}
0&0&0&1&0 \\
0&\frac{(q-2)(q^{2m-2}-1)}{2}&\frac{(q-2)(q^{2m-2}-1)}{2}&\frac{(q-2)q^{2m-2}}{2}&0\\
0&\frac{q(q^{2m-2}-1)}{2}&\frac{q(q^{2m-2}-1)}{2}&\frac{q^{2m-1}}{2}&0 \\
q(q^{2m-2}-1)&q^{2m-2}-1&q^{2m-2}-1&q^{2m-2}-2q+1&q(q^{2m-2}-1) \\
0&0&0&q-2&0
\end{bmatrix}, \\
B_4&=\begin{bmatrix}
0&0&0&0&1 \\
0&\frac{q-4}{2}&\frac{q-2}{2}&0&0 \\
0&\frac{q}{2}&\frac{q-2}{2}&0&0 \\
0&0&0&q-2&0 \\
q-2&0&0&0&q-3
\end{bmatrix}.
\end{align*}

\begin{lem}\label{lem:NS}
The algebra $N(W)$ is symmetric.
\end{lem}
\begin{proof}
Suppose that $N(W)$ is not symmetric. Then
by \cite[Proposition 6(i)]{JMN}, there exists $(b,c)\in X^2$ with $b\neq c$
such that
\[
\sum_{x\in X}\frac{W_{x,b}^2}{W_{x,c}^2}=0.
\]
This is equivalent to
\[
\sum_{j,k}
p_{jk}^i \frac{w_j^2}{w_k^2}=0
\]
for some $i\in\{1,2,3,4\}$. Using the notation (\ref{aij0802}), we have
\begin{align}
\sum_{j,k}p_{jk}^i \frac{w_j^2}{w_k^2}
&=
\sum_{j<k}p_{jk}^i 
\left(\frac{w_j^2}{w_k^2}
+\frac{w_k^2}{w_j^2}\right)
+\sum_{j=0}^4 p_{jj}^i 
\nnexteq
\sum_{j<k}p_{jk}^i 
\left(\left(\frac{w_j}{w_k}
+\frac{w_k}{w_j}\right)^2-2\right)
+\sum_{j=0}^4 p_{jj}^i 
\nnexteq
\sum_{j<k}p_{jk}^i (a_{j,k}^2-2)
+\sum_{j=0}^4 p_{jj}^i. 
\label{eq:NS}
\end{align}
It can be verified by computer that (\ref{eq:NS})
is nonzero for each of the cases (i)--(ii) in Theorem~\ref{thm:main}.
\end{proof}

Since $N(W)$ is symmetric, the adjacency matrices of $N(W)$ are the
$(0,1)$-ma\-trices representing the connected components of 
the Jones graph defined as follows
(see \cite[Sect.~3.3]{JMN}).
The {\em Jones graph} of a type-II matrix $W\in M_X(\C)$ is the graph
with vertex set $X^2$ such that two distinct vertices $(a,b)$
and $(c,d)$ are adjacent whenever $\langle Y_{ab}, Y_{cd}\rangle\neq0$,
where $\langle \;,\;\rangle$ denotes the ordinary (not Hermitian)
scalar product.

\begin{proof}[Proof of Theorem~{\rm\ref{thm:main2}}]

We claim that $(x,y)$ and $(x,z)$ belong to the same connected component
in the Jones graph whenever $(x,y),(y,z),(z,x)\in R_4$.
Indeed, if $(x,y)$ and $(x,z)$ belong to different connected components, then
$(y,x)$ and $(z,x)$ belong to different connected components by Lemma~\ref{lem:NS}.
In particular, 
\[
\langle Y_{xy},Y_{xz}\rangle=\langle Y_{yx},Y_{zx}\rangle=0.
\]
Let
\[
c_{i,j,k}=|\{u\in X\mid (x,u)\in R_i,\;(y,u)\in R_j,\;(z,u)\in R_k\}|.
\]
Since $p_{1,3}^4=p_{2,3}^4=0$, we have
\begin{equation} \label{cijk}
c_{1,j,k}+c_{2,j,k}=
c_{j,1,k}+c_{j,2,k}=
c_{j,k,1}+c_{j,k,2}=p_{j,k}^4
\end{equation}
for $j,k\in\{1,2\}$. Then we have
\begin{equation} \label{eq:1026}
\sum_{i,j,k=0}^4 c_{i,j,k}\frac{w_i^2}{w_jw_k}=
\sum_{i,j,k=0}^4 c_{i,j,k}\frac{w_jw_k}{w_i^2}=0.
\end{equation}
Since the rank of the coefficient matrix in (\ref{cijk}) is $7$,
we have one degree of freedom in (\ref{cijk}).
Combining (\ref{cijk}) and (\ref{eq:1026}),
it can be verified by computer that these conditions
give rise to a polynomial equation in $q$ which has no solution in positive integers $q\geq 4$ 
for each of the cases (i) and (ii) in Theorem~\ref{thm:main}.

Therefore, we have proved the claim.
This, together with Lemma~\ref{lem:NS}, implies that, for each equivalence class $C$ of the
equivalence relation $R_0\cup R_4$,
$(C\times C)\cap R_4$ belongs to the
same connected component in the Jones graph.

Let $C$ and $C'$ be two distinct equivalence classes of $R_0\cup R_4$.
We claim that,
for any $(x,z)\in C\times C'$, there exist $y\in C$ such that
$\langle Y_{xy},Y_{xz}\rangle\neq0$, and
there exist $y'\in C'$ such that
$\langle Y_{y'z},Y_{xz}\rangle\neq0$.

Suppose
$(x,z)\in R_1$ and 
$\langle Y_{xy},Y_{xz}\rangle=0$ for all $y\in R_4(x)$. Then
\begin{align*}
0&=\sum_{y\in R_4(x)}\langle Y_{xy},Y_{xz}\rangle
\nexteq
\sum_{y\in R_4(x)}\sum_{u\in X} (Y_{xy})_u(Y_{xz})_u
\nexteq
\sum_{y\in R_4(x)}\sum_{u\in X} \frac{W_{xu}^2}{W_{yu}W_{zu}}
\nexteq
\sum_{y\in R_4(x)}\sum_{i,j=0}^4 \sum_{u\in R_i(x)\cap R_j(z)}
\frac{W_{xu}^2}{W_{yu}W_{zu}}
\nexteq
\sum_{i,j=0}^4 \sum_{u\in R_i(x)\cap R_j(z)}
\sum_{k=0}^4 \sum_{y\in R_4(x)\cap R_k(u)}
\frac{w_i^2}{w_kw_j}
\nexteq
\sum_{i,j=0}^4 \sum_{u\in R_i(x)\cap R_j(z)}
\sum_{k=0}^4 p_{4k}^i \frac{w_i^2}{w_kw_j}
\nexteq
\sum_{i,j,k=0}^4 p_{ij}^1
p_{4k}^i \frac{w_i^2}{w_kw_j}.
\end{align*}
It can be verified by computer that this leads to a polynomial
equation in $q$ which has no solution in positive integers $q\geq4$.
Set $\ell\in\{2,3\}$.
Similarly, suppose $(x,z)\in R_{\ell}$ and
$\langle Y_{xy},Y_{xz}\rangle=0$ for all $y\in C$. Then
\[
\sum_{i,j,k=0}^4 p_{ij}^{\ell}
p_{4k}^i \frac{w_i^2}{w_kw_j}=0,
\]
and again this leads to a contradiction.
Thus, there exists $y\in C$
such that $\langle Y_{xy},Y_{xz}\rangle\neq0$.
Switching the role of $x$ and $z$, we see that there exists
$y'\in C'$ such that $\langle Y_{y'z},Y_{xz}\rangle\neq0$.
Therefore, we have proved the claim.

Since $C$ and $C'$ are arbitrary,
the claim shows that, in the Jones graph, $R_4$ is contained in a single connected component,
and that every element $(x,z)\in R_1\cup R_2\cup R_3$ is adjacent to an element
of $R_4$. Thus,
$R_1\cup R_2\cup R_3\cup R_4$ is a connected component of the Jones graph. Therefore, $\dim N(W)=2$.
\end{proof}

\newpage

\appendix
\section{Verification by Magma}

\section*{Proof of Theorem~\ref{thm:main}}
\begin{verbatim}
d:=4;
d2s:=&cat[[[i,j]:j in [i+1..d]]:i in [0..d-1]];
d2:=[Seqset(s):s in d2s];
R:=PolynomialRing(Rationals(),#d2+3);
X:=func<i,j|R.Position(d2,{i,j})>;
q:=R.(#d2+1);
r:=R.(#d2+2);
nz1:=R.(#d2+3);
NZ1:=nz1*(q-1)-1;
qm:=q*r;

g:=func<i,j,k|X(i,j)^2+X(i,k)^2+X(j,k)^2-X(i,j)*X(i,k)*X(j,k)-4>;
h:=func<i,j,k,l|(X(k,l)^2-4)*X(i,j)
 -X(k,l)*(X(k,i)*X(l,j)+X(k,j)*X(l,i))
 +2*(X(k,i)*X(k,j)+X(l,i)*X(l,j))>;

eigenP:=Matrix(R,5,5,[
1,1/2*qm*r*(q-2),1/2*qm^2,q*(r^2-1),q-2,
1,1/2*r*(q-2),1/2*qm,-(r+1)*(q-1),q-2,
1,-1/2*r*(q-2),-1/2*qm,(r-1)*(q-1),q-2,
1,1/2*qm,-1/2*qm,0,-1,
1,-1/2*qm,1/2*qm,0,-1
]);
P:=func<i,j|eigenP[i+1,j+1]>;
n:=&+[P(0,i):i in [0..d]];
n eq qm^2-1;

e:=func<i|-n+&+[P(i,j)^2:j in [0..d]]
 +&+[P(i,j[1])*P(i,j[2])*X(j[1],j[2]):j in d2s]>; //eq:21
s3:=[Setseq(x):x in Subsets({0..d},3)];
eq7:=[g(i[1],i[2],i[3]):i in s3] cat
 [h(0^i,1^i,2^i,3^i):i in Sym({0..d})] cat
 [e(i):i in [1..d]];
I:=ideal<R|eq7>;
\end{verbatim}

\subsection*{Proof of Lemma~\ref{lem:a01}}
\begin{verbatim}
I1:=ideal<R|I,X(0,4)-2>;
pa12:=qm^2*X(1,2)+2*(qm^2-2);
pa12 in I1; //Lemma 3
\end{verbatim}

\subsection*{Proof of Lemma~\ref{lem:a02eq2}}
\begin{verbatim}
I2:=ideal<R|I1,X(0,1)-2>;
(r^2-1)^2*(X(0,3)-2)^2 in I2;
\end{verbatim}

\subsection*{Proof of Lemma~\ref{lem:a02d}}
\begin{verbatim}
pa01:=qm*r*(qm*r+q-2)*X(0,1)-2*(qm^2*r^2-qm^2-2*qm*r+2);
I3:=ideal<R|I1,pa01>;
pa02:=q*(qm*r+q-2)*X(0,2)+2*qm^2-2*q^2+4*q-4;
ff:=(qm^2-1)*(148*r^4-8103*r^2+8214*q^2-46102*q+42957)
*((q+1)*(2*q-1)*qm^4*r^6+(2*q+1)*(q^2-7*q+4)*qm^3*r^5
 +(5*q^3+27*q^2-22*q-4)*qm^2*r^4-(q^3+11*q^2+46*q-56)*qm^2*r^2
 +8*(q+6)*(q-1)*q*r^2-16*q+16);
pa02*ff in I3;

I4:=ideal<R|I3,pa02>;
(r^2-1)*(X(0,3)-2)^2 in I4;
//Total time: 88023.889 seconds, Total memory usage: 253.50MB
\end{verbatim}

\subsection*{Proof of Theorem~\ref{thm:main}}
\begin{verbatim}
Pqr:=PolynomialRing(Rationals(),2);
Fqr<q,r>:=FieldOfFractions(Pqr);
Rqr<z1>:=PolynomialRing(Fqr);
qm:=q*r;
a01:=2*(qm^2*r^2-(q+2)*qm*r+2)/(qm*r*(qm*r+q-2));
a02:=-2*(qm^2-q^2+2*q-2)/(q*(qm*r+q-2));
a12:=-2*(qm^2-2)/(qm^2);
F<w1>:=FieldOfFractions(Rqr/ideal<Rqr|z1^2-a01*z1+1>);
w1^2-a01*w1+1 eq 0;
w2:=-((q-1)*qm*r*w1+qm*r+q-2)/((qm*r-1)*q);
w1/w2+w2/w1 eq a12;
w2+1/w2 eq a02;

d:=4;
d2s:=&cat[[[i,j]:j in [i+1..d]]:i in [0..d-1]];
d2:=[Seqset(s):s in d2s];

R:=PolynomialRing(F,#d2);
X:=func<i,j|R.Position(d2,{i,j})>;
g:=func<i,j,k|X(i,j)^2+X(i,k)^2+X(j,k)^2-X(i,j)*X(i,k)*X(j,k)-4>;
h:=func<i,j,k,l|(X(k,l)^2-4)*X(i,j)
 -X(k,l)*(X(k,i)*X(l,j)+X(k,j)*X(l,i))
 +2*(X(k,i)*X(k,j)+X(l,i)*X(l,j))>;

eigenP:=Matrix(F,5,5,[
1,1/2*qm*r*(q-2),1/2*qm^2,q*(r^2-1),q-2,
1,1/2*r*(q-2),1/2*qm,-(r+1)*(q-1),q-2,
1,-1/2*r*(q-2),-1/2*qm,(r-1)*(q-1),q-2,
1,1/2*qm,-1/2*qm,0,-1,
1,-1/2*qm,1/2*qm,0,-1
]);
P:=func<i,j|eigenP[i+1,j+1]>;
n:=&+[P(0,i):i in [0..d]];
n eq qm^2-1;

e:=func<i|-n+&+[P(i,j)^2:j in [0..d]]
 +&+[P(i,j[1])*P(i,j[2])*X(j[1],j[2]):j in d2s]>;
s3:=[Setseq(x):x in Subsets({0..d},3)];
eq7:=[g(i[1],i[2],i[3]):i in s3] cat
 [h(0^i,1^i,2^i,3^i):i in Sym({0..d})] cat
 [e(i):i in [1..d]];

subs1:=[a01,a02,2,2,a12,a01,a01,a02,a02,2];
&and[Evaluate(f,subs1) eq 0:f in eq7];
//Total time: 0.440 seconds, Total memory usage: 32.09MB
\end{verbatim}

\section*{Proof of Theorem~\ref{thm:main2}}
Calculation of the intersection matrices $\{B_i\}_{i=0}^4$:
\begin{verbatim}
P<c111,c112,c121,c122,c211,c212,c221,c222,w1,w2,q,r>
 :=PolynomialRing(Rationals(),12);
F:=FieldOfFractions(P);
qm:=q*r;
n:=qm^2-1;

eigenP:=Matrix(F,5,5,[
1,1/2*qm*r*(q-2),1/2*qm^2,q*(r^2-1),q-2,
1,1/2*r*(q-2),1/2*qm,-(r+1)*(q-1),q-2,
1,-1/2*r*(q-2),-1/2*qm,(r-1)*(q-1),q-2,
1,1/2*qm,-1/2*qm,0,-1,
1,-1/2*qm,1/2*qm,0,-1
]);

intersectionMatrices:=function(P)
  d1:=Nrows(P);
  n:=&+Eltseq(P[1]);
  Q:=n*P^(-1);
   return [ Matrix(Parent(Q[1][1]),d1,d1,
      [ [ 1/(n*P[1,k])*&+[ Q[1,l]*P[l,i]*P[l,j]*P[l,k] : l in [1..d1] ]
      : k in [1..d1] ] : j in [1..d1] ]
      ) : i in [1..d1] ];
end function;

B1:=Matrix(F,5,5,[0,1,0,0,0,
qm*r*(q-2)/2,r^2*(q-2)^2/4,r^2*(q-2)^2/4,r^2*(q-2)^2/4,(q-4)*qm*r/4,
0,(q-2)*qm*r/4,(q-2)*qm*r/4,(q-2)*qm*r/4,qm^2/4,
0,(q-2)*(r^2-1)/2,(q-2)*(r^2-1)/2,(q-2)*r^2/2,0,
0,1/2*q-2,1/2*q-1,0,0]);

B2:=Matrix(F,5,5,[0,0,1,0,0,
0,(q-2)*qm*r/4,(q-2)*qm*r/4,(q-2)*qm*r/4,qm^2/4,
qm^2/2,qm^2/4,qm^2/4,qm^2/4,qm^2/4,
0,q*(r^2-1)/2,q*(r^2-1)/2,1/2*qm*r,0,
0,1/2*q,1/2*q-1,0,0]);

B3:=Matrix(F,5,5,[0,0,0,1,0,
0,(q-2)*(r^2-1)/2,(q-2)*(r^2-1)/2,(q-2)*r^2/2,0,
0,q*(r^2-1)/2,q*(r^2-1)/2,1/2*qm*r,0,
q*(r^2-1),r^2-1,r^2-1,r^2-2*q+1,q*(r^2-1),
0,0,0,q-2,0]);

B4:=Matrix(F,5,5,[0,0,0,0,1,
0,1/2*q-2,1/2*q-1,0,0,
0,1/2*q,1/2*q-1,0,0,
0,0,0,q-2,0,
q-2,0,0,0,q-3]);

BB:=[ScalarMatrix(5,F!1),B1,B2,B3,B4];
BB eq intersectionMatrices(eigenP);
pijk:=func<i,j,k|P!BB[i+1][j+1,k+1]>;
\end{verbatim}

\subsection*{Proof of Lemma~\ref{lem:NS}}
\begin{verbatim}
isSymNbas:=function(ajk)
 aijs:=[[ajk[1],ajk[2],ajk[3],ajk[4]],
  		[1,ajk[5],ajk[6],ajk[7]],
		[1,1,ajk[8],ajk[9]],
		[1,1,1,ajk[10]]];
 aij:=func<i,j|aijs[i+1,j]>;
 ff:=[F|&+[pijk(j,k,i)*(aij(j,k)^2-2):j,k in [0..4]|j lt k]
  +&+[pijk(j,j,i):j in [0..4]]:i in [1..4]];
 return [Numerator(ff[i]):i in [1..4]];
end function;

x02:=-(2*q^2*r^2-4)/(q^2*r^2);
aa:=[2,x02,2,2,x02,2,2,x02,x02,2];
isSymNbas(aa) eq [ (qm^2-1)*(qm^2-4) :i in [1..4]];

a01:=2*(qm^2*r^2-(q+2)*qm*r+2)/(qm*r*(qm*r+q-2));
a02:=-2*(qm^2-q^2+2*q-2)/(q*(qm*r+q-2));
a12:=-2*(qm^2-2)/(qm^2);
aa:=[a01,a02,2,2,a12,a01,a01,a02,a02,2];
pp:=qm^5*r+2*(q^2-10*q+14)*qm^3*r+
 q*(q-2)*(q^3-2*q^2+8*q+16)*r^2-4*(q-2)*(q^2-2*q+4);
isSymNbas(aa) eq [ (qm^2-1)*pp : i in [1..3] ] cat
 [ (qm^2-1)*(qm^2-4) ];
\end{verbatim}

\subsection*{Proof of Theorem~\ref{thm:main2}}

\subsection*{The first claim:}
\begin{verbatim}
varname:=[[i,j,k]:i,j,k in [1,2]];
c:=func<i,j,k|R.Position(varname,[i,j,k])>;
w0:=1;

cijk:=function(i,j,k)
  if 0 in {i,j,k} then
      if [i,j,k] in {[0,4,4],[4,0,4],[4,4,0]} then
      return 1;
      else 
      return 0;
      end if;
  else
     if 3 in {i,j,k} then
       if {3} eq {i,j,k} then
       return pijk(3,3,4);
       else
       return 0;
       end if;
     else
       if 4 in {i,j,k} then
	     if {4} eq {i,j,k} then
	     return pijk(4,4,4)-1;
	     else
	     return 0;
	     end if;
	   else
	   return c(i,j,k);
       end if;
     end if;
  end if;
end function;

fx:=[cijk(1,j,k)+cijk(2,j,k)-pijk(j,k,4):j,k in [1,2]];
fy:=[cijk(j,1,k)+cijk(j,2,k)-pijk(j,k,4):j,k in [1,2]];
fz:=[cijk(j,k,1)+cijk(j,k,2)-pijk(j,k,4):j,k in [1,2]];
fxyz:=fx cat fy cat fz;

a02N1:=-(2*q^2*r^2-4);
a02D1:=q^2*r^2;
a021:=a02N1/a02D1;
fa21:=a02D1*w2^2-a02N1*w2+1;
ww1:=[w0,1,w2,1,1];

alpha1:=func<i|ww1[i+1]>;
ff1:=&+[cijk(i,j,k)*alpha1(i)^2/(alpha1(j)*alpha1(k))
 :i,j,k in [0..4]];
gg1:=&+[cijk(i,j,k)*alpha1(j)*alpha1(k)/alpha1(i)^2
 :i,j,k in [0..4]];
I1:=ideal<R|[fa21,Numerator(ff1),Numerator(gg1)] cat fxyz>;
(qm^2-1)*(5*qm^6-90*qm^4+313*qm^2-128) in I1;
IsIrreducible(5*qm^6-90*qm^4+313*qm^2-128);

a01N2:=2*(qm^2*r^2-(q+2)*qm*r+2);
a01D2:=(qm*r+q-2)*qm*r;
a01:=a01N2/a01D2;
a022:=-2*(qm^2-q^2+2*q-2)/(q*(qm*r+q-2));
a12:=-2*(qm^2-2)/(qm^2);
fa1:=a01D2*w1^2-a01N2*w1+a01D2;
w2:=(w1^2-1)/(a12*w1-a02);
ww2:=[w0,w1,w2,1,1];

alpha2:=func<i|ww2[i+1]>;
ff2:=&+[cijk(i,j,k)*alpha2(i)^2/(alpha2(j)*alpha2(k))
 :i,j,k in [0..4]];
gg2:=&+[cijk(i,j,k)*alpha2(j)*alpha2(k)/alpha2(i)^2
 :i,j,k in [0..4]];
I2:=ideal<R|[fa1,Numerator(ff2),Numerator(gg2)] cat fxyz>;
Basis(EliminationIdeal(I2,{q,r}))
 eq [qm^10*(qm^2-1)^3*(qm*r+q-2)^4*(qm*r-1)^5];
\end{verbatim}

\subsection*{The second claim:}
\begin{verbatim}
tl:=function(l)
 return &+[pijk(i,j,l)*pijk(4,k,i)*alpha1(i)^2/(alpha1(j)*alpha1(k))
  :i,j,k in [0..4]]; 
end function;
&and[ (q-2)*(qm^2-1)*(5*qm^6-90*qm^4+313*qm^2-128) in 
 ideal<R|[fa21,Numerator(tl(l))]> : l in [1..3] ];
 
sl:=function(l)
 return &+[pijk(i,j,1)*pijk(4,k,i)*alpha2(i)^2/(alpha2(j)*alpha2(k))
  :i,j,k in [0..4]]; 
end function;
&and[ qm^7*r*(q-2)*(qm^2-1)^3*(qm*r-1)^5*(qm*r+q-2)^5 in 
 ideal<R|[fa1,Numerator(sl(l))]> : l in [1..3] ];
//Total time: 34.890 seconds, Total memory usage: 82.78MB
\end{verbatim}

\end{document}